\documentclass[12pt,a4paper]{article}
\usepackage{latexsym}
\usepackage{amsfonts}
\usepackage{latexsym}
\usepackage{amsmath,amsthm}
\newtheorem{theorem}{Theorem}
\newtheorem*{Kakutani}{Kakutani fixed point theorem}
\newtheorem*{Alaoglu}{Banach-Alaoglu theorem}{\bf}{\it}
\newtheorem{lemma}{Lemma}
\newtheorem{definition}{Definition}

\newcommand{\R}{\mathbf R}

\newcommand\E{\mathbf E}
\newcommand{\f}{f : X\to\R^n}
\newcommand{\g}{g : X\to\R^{m}}

\newcommand{\norm}[1]{\lVert#1\rVert}
\newcommand{\lhad}[2]{#2^{[#1]}_{-} }
\newcommand{\lsubd}[2]{\partial^{(#1)}_{-} #2}
\newcommand{\ldini}[1]{#1^{(1)}_{-} }
\newcommand{\be}{\begin{equation}}
\newcommand{\ee}{\end{equation}}

\newcommand\pr{\prime}
\newcommand\dom{\rm dom}

\title{First and Second Order Necessary and Sufficient Optimality Conditions of Fritz John Type for Vector Problems over Cones} 
\author{Vsevolod I. Ivanov\thanks{Email: vsevolodivanov@yahoo.com
\vspace{6pt} }
\\\vspace{6pt}{\em{\small Department of Mathematics, Technical University of Varna, 9010 Varna, Bulgaria} }
}
\date{\today}

\begin{document}
\maketitle

\begin{abstract}

\smallskip

In this paper, we obtain a new proof of Fritz John necessary optimality conditions for vector problems applying Kakutani fixed point theorem and Hadamard directional derivative. We also derive a similar proof of  second-order Fritz John necessary optimality conditions. Sufficient conditions for weak global efficiency with generalized convex functions and local efficiency are provided.
 
2010 {\it Mathematics Subject Classification.}
90C26, 90C25, 90C29, 26B25

\smallskip  
{\it Key words and phrases.}
optimality conditions, vector problems, cone optimization, fixed point theorems

\end{abstract}





\section{Introduction}
In the optimality conditions for the problems with inequality or equality constraints or both appear Lagrange multipliers.
Optimality conditions of Fritz John type play important role in optimization. After appearing  Fritz John \cite{FJ48} and Karush-Kuhn-Tucker \cite{kar39,kt51} necessary conditions optimization separated from mathematical analysis as a relatively self-dependent science. Various proofs were published after the initial one by Fritz John 
\cite{arr61,bel69,bir07,bre09,bre12,can66,Cottle,DaCunha,hal74,mcshane,MF67,McCormic,por80,sla50}.
Such conditions are usually proved by showing the inconsistency of a system with inequalities and equations (primal conditions), by separations of sets, by theorems of the alternatives, by duality and so on.

In this paper, we obtain necessary and sufficient Fritz John  optimality conditions  for weak minimum for the following vector optimization problem, denoted by (P):

\medskip
Minimize $f(x)$ with respect to the cone $C$ 

subject to $g(x)\in -K$,
\medskip

\noindent
where $C\subset\R^n$, $K\subset\R^m$ are some cones. This problem includes as a particular case the problem with inequality and equality constraints.
We propose another proof, which is based on Kakutani's fixed point theorem. It is quite different from the proofs in 
\cite{can66,DaCunha,hal74}, where Brouwer's fixed point theorem were applied. Halkin's proof \cite{hal74} uses implicit functions, and Canon, Cullum, Polak's \cite{can66} proof uses separation of a ray and a cone. Brouwer's theorem were used to prove this separation. Similar arguments were applied by Da Cunha, Polak \cite{DaCunha}, but they treated vector problems. Really, Kakutani's fixed point theorem concern set-valued maps and it is a generalization of  Brouwer's fixed point theorem. On the other hand, the published proofs, where  Brouwer's  theorem were applied are complicated, but our proof is very simple.
We derive second-order necessary conditions for weak local minimum applying again Kakutani's fixed point theorem. In these two theorems, we consider vector problems such that minimization is with respect some cone, eventually different from positive orthant. The constraints are connected with another cone.

We also obtain  sufficient global optimality conditions for the problem (P) with generalized convex functions. In the first-order ones the objective function is pseudoconvex and the constraint function is strictly pseudoconvex. Pseudoconvex functions were introduced for optimization problems in \cite{man65}. We extend  this notion to minimization problems with respect to a cone. In the second-order conditions,  the objective function is second-order pseudoconvex and the constraint function is second-order strictly pseudoconvex. Second-order pseudoconvex functions were introduced in \cite{jmaa2008}. We generalize this notion to minimization problems with respect to a cone.
      
At last, we derive first and second-order sufficient optimality conditions for isolated weak local minimum. In these conditions, the nonstrict inequalities are replaced by strict ones.

\section{Fritz John theorem for vector problems over cones}
In this section, we obtain first-order necessary conditions for weak local minimum of Fritz John type.

Consider the vector optimization problem (P),
%
%
%
where $\f$ and $\g$ are given  vector functions defined on some open set $X\subset\R^s$. The cones $C\in\R^n$ and $K\in\R^{m}$ are given closed convex pointed ones, whose vertexes are the origins of the respective spaces.

Denote by $S$ the feasible set, that is
\[
S:=\{x\in X\mid g(x)\in -K\}.
\]

\begin{definition}[\cite{jahn}]
A feasible point $\bar x$ is called a weak local minimizer, iff there exists a neighborhood $N\ni\bar x$ such that there is no another feasible point $x\in S\cap N$ with the property $f(x)\in f(\bar x)-{\rm int}(C)$.
\end{definition}

Let $C$ be a poined cone with nonempty interior ${\rm int}(C)$, whose vertex is the origin. Denote its positive polar cone by $C^*$ 
\[
C^*:=\{\lambda\in R^n\mid\lambda\cdot x\ge 0\textrm{ for all }x\in C\},
\]

\begin{definition}[\cite{DemRub}]
The lower Dini directional derivative of a given function $f:\E\to\R\cup\{+\infty\}$ at the point $x\in\dom$ $f$ in direction $u\in\E$ is defined as follows:
\[
\ldini f  (x;u)=\liminf_{t\downarrow 0}\,t^{-1}[f(x+t u)-f(x)].
\]
\end{definition}

\begin{definition}[\cite{DemRub}]
The lower Hadamard directional derivative of a given function $f:\E\to\R\cup\{+\infty\}$ at the point $x\in\dom$ $f$ in direction $u\in\E$ is defined as follows:
\[
\lhad 1 f (x;u)=\liminf_{t\downarrow 0,u^\pr\to u}\,t^{-1}[f(x+t u^\pr)-f(x)].
\]
\end{definition}

Suppose that $\bar x\in S$ is a weak local minimizer. Then, consider the function 
\[
F(x):=\sup\{\lambda\cdot[f(x)-f(\bar x)]+\mu\cdot g(x)\mid (\lambda,\mu)\in\Lambda\},
\]
where $\Lambda:=\{(\lambda,\mu)\mid\lambda\in C^*,\;\mu\in K^*,\;\sum_{i=1}^n|\lambda_i|+\sum_{j=1}^m|\mu_j|=1\}$.

\begin{lemma}[\cite{na2015}]\label{na.F}
Suppose that $\bar x\in S$ is a weak local minimizer. Then, there exists a neighborhood $N\ni\bar x$ such that $F(x)\geq F(\bar x)=0$ for all $x\in N$.
\end{lemma}

\begin{lemma}[\cite{na2015}]\label{NA}
Suppose that $\bar x\in S$ is a weak local minimizer. Then
\[
\lhad 1 F(\bar x;u)\ge 0\quad\textrm{for all}\quad u\in\R^s. 
\]
\end{lemma}

The followng lemma is well known \cite{ggt2004}. 

\begin{lemma}[\cite{jahn}]\label{lema2}
Let $C\subset\R^n$ be a closed convex cone and $x\in C$. Then $x\in {\rm int}(C)$ if and only if $\lambda\cdot x>0$ for all $\lambda\in C^*$ with 
$\lambda\ne 0$.
\end{lemma}

\begin{Kakutani}
Let $X$ be a nonempty convex and compact set in a Banach space. Suppose that $F:X\to 2^X$  be a multi-valued map. For every $x\in X$ the set $F(x)$ is nonempty convex and let $F$ be closed. Then $F$ has a fixed point, that is there exists $x_0\in X$ with $x_0\in F(x_0)$.
\end{Kakutani}

Let us consider the general problem (P), where the cones $C$ and $K$ do not necessarily coincide with the positive orthants of the respective spaces.

\begin{theorem}\label{th2}
Let $\bar x$ be a weak local minimizer, $f$ and $g$ be Fr\'echet  differentiable  at $\bar x$ vector-valued functions. 
Then, there exist vectors $\lambda\in C^*$ and $\mu\in K^*$ such that  $(\lambda,\mu)\ne (0,0)$, $\mu\cdot g(\bar x)=0$ and
\[
\lambda\nabla f(\bar x)+\mu\nabla g(\bar x)=0.
\]
If the cone $K$ has nonempty interior and $\bar x$ belongs to its interior, then there exist multipliers with $\mu=0$.
\end{theorem}
\begin{proof}
Let $u\in\R^s$ be a fixed direction. We prove that there exist $\lambda\in C^*$ and $\mu\in K^*$ such that $(\lambda,\mu)\in\Lambda$ and
\be\label{5}
[\lambda\cdot\nabla f(\bar x)+\mu\cdot\nabla g(\bar x)]u\ge 0.
\ee
We obtain from Lemma \ref{NA} taking into account the equality $F(\bar x)=0$ that
\[
\lhad 1 F(\bar x;u)=\liminf_{t\downarrow 0,u^\pr\to u}\,\max_{(\lambda,\mu)\in\Lambda} [\lambda\cdot\frac{f(\bar x+t u^\pr)-f(\bar x)}{t}+\mu\cdot\frac{g(\bar x+t u^\pr)}{t}]\ge 0.
\] 
We have $\mu\cdot g(\bar x)\le 0$ for all $\mu\in K^*$. Then, it follows from here that
\[
\liminf_{t\downarrow 0,u^\pr\to u}\,\max_{(\lambda,\mu)\in\Lambda} [\lambda\cdot\frac{f(\bar x+t u^\pr)-f(\bar x)}{t}+\mu\cdot\frac{g(\bar x+t u^\pr)-g(\bar x)}{t}]\ge 0.
\] 
On the other hand the lower Dini directional derivative is greater or equal to the lower Hadamard directional derivative. Therefore
\[
\liminf_{t\downarrow 0}\,\max_{(\lambda,\mu)\in\Lambda} [\lambda\cdot\frac{f(\bar x+t u)-f(\bar x)}{t}+\mu\cdot\frac{g(\bar x+t u)-g(\bar x)}{t}]\ge 0.
\]
Suppose that the lower limit is approached over the sequence $t_k$. Therefore,
\[
\lim_{k\to +\infty}\,\max_{(\lambda,\mu)\in\Lambda} [\lambda\cdot\frac{f(\bar x+t_k u)-f(\bar x)}{t_k}+\mu\cdot\frac{g(\bar x+t_k u)-g(\bar x)}{t_k}]\ge 0.
\]
Suppose that the above maximum is attained over $\lambda_k$ and $\mu_k$. Then
\be\label{3}
\lim_{k\to +\infty}\, [\lambda_k\frac{f(\bar x+t_k u)-f(\bar x)}{t_k}+\mu_k\frac{g(\bar x+t_k u)-g(\bar x)}{t_k}]\ge 0.
\ee
Using that the set $\Lambda$ is compact, then we can suppose that $\lambda_k\to\lambda$, $\mu_k\to\mu$. The functions  $f$ and $g$ are  Fr\'echet differentiable. Then, taking the limits when $k$ approaches $+\infty$ we obtain that 
\be\label{12}
[\lambda\cdot\nabla f(\bar x)+\mu\cdot\nabla g(\bar x)]u\ge 0.
\ee
It follows from here that for every $u\in\R^s$ there exists $(\lambda,\mu)$ such that the inequality (\ref{12}) is satisfied.
Let us consider the multivalued map $A(u)$, which associates any direction $u$ with the pair $(\lambda,\mu)$ satisfying inequality (\ref{12}), that is
 $(\lambda,\mu)\in A(u)$. Consider also the map 
\[
B(\lambda,\mu)=-\sum_{i=1}^n \lambda_i\nabla f_i(\bar x)-\sum_{j=1}^m\mu_j\nabla g_j(\bar x)
\]
The map $A(u)$ is nonempty, closed and convex-valued for every direction $u$. $B$ is linear. Then, by Kakutani fixed point theorem the composite map $B\cdot A$ has a fixed point. Therefore, there exists a direction $u_0$ such that $u_0\in B\cdot A(u_0)$.
Then,
\[
[\sum_{i=1}^n \lambda_i\nabla f_i(\bar x)+\sum_{j=1}^m \mu_j\nabla g_j(\bar x)](u_0)\ge 0
\]
where 
\[
u_0=-\sum_{i=1}^n \lambda_i\nabla f_i(\bar x)-\sum_{j=1}^m \mu_j\nabla g_j(\bar x)
\]
It follows from here that $\norm{u_0}^2\le 0$. Therefore $u_0=0$, that is
\[
\sum_{i=1}^n \lambda_i\nabla f_i(\bar x)+\sum_{j=1}^m\mu_j\nabla g_j(\bar x)=0.
\]
By the definition of the set $\Lambda$ we have $(\lambda,\mu)\ne (0,0)$. It follows from the definition of the function $F$ that $\mu\cdot g(\bar x)=0$.
We conclude from the definition of the derivative $\lhad 1 F(\bar x;u)$ that if  ${\rm int}(K)\ne\emptyset$ and $\bar x$ belongs to the interior of $K$, then there exist multipliers with $\mu=0$.
\end{proof}


\begin{Alaoglu}[\cite{rud91}, Section 3.15]
Let $\E$ be a normed vector space, and let $\E^*$ be its dual. Then the closed unit ball $B$ of $\E^*$, that is 
\[
B=\{\xi\in\E^*\mid\norm{\xi}\le 1\},
\]
is compact in the weak-$*$ topology.
\end{Alaoglu}

Consider the so called sequence space. It is usually denoted by $l_\infty$ and consists of all infinite bounded sequences ${x_n}$ endowed  with the norm 
\[\norm{x}_\infty=\sup_n |x_n|\].

\begin{theorem}
Let $C$ and $K$ be nonempty convex cones in the infinite-dimensional sequence space $l_\infty$ .   Suppose that $\bar x$ is a weak local minimizer, $f$ and $g$ are continuously 
Fr\' echet  differentiable vector-valued functions at $\bar x$.  Then, there exist vectors $\lambda\in C^*$ and $\mu\in K^*$ such that  $(\lambda,\mu)\ne 0$ and
\[
\sum_{i=1}^\infty\lambda_i\nabla f_i(\bar x)+\sum_{j=1}^\infty\mu_j\nabla g_j(\bar x)=0,\; \mu\cdot g(\bar x)=0.
\]
\end{theorem}
\begin{proof}
The theorem can be proved applying the arguments of Theorem \ref{th2}.
\end{proof}

\section{Second-order necessary 
conditions of Fritz John type} 

In this section, we obtain second-order necessary conditions for weak local minimum of Fritz John type.

\begin{definition}[\cite{RocWet}]
Lower subdifferential of Hadamard type for the function $f:\E\to\R\cup\{+\infty\}$ at the point $x\in{\dom}\, f$ is defined as follows:
\[
\lsubd 1  f(x)=\{ x^*\in L^1(\E)\mid x^*(u)\le\lhad 1 f (x;u)\quad\textrm{for arbitrary direction}\quad u\in\E\}.
\]
\end{definition}


\begin{definition}[\cite{RocWet}]\label{na.df-2had}
Let $f:\E\to\R\cup\{+\infty\}$ be arbitrary proper real function. Suppose that $x^*_1$ is a fixed point from the lower Hadamard subdifferential $\lsubd 1 f (x)$ at the point $x\in{\dom}\, f$. Then the second-order lower derivative of Hadamard type for the function $f$ at the point $x\in{\dom}\, f$ in direction $u\in\E$ can be defined as follows:
\[
\lhad 2  f(x;x^*_1;u)=\liminf_{t\downarrow 0,u^\pr\to u}\,
2t^{-2}[f(x+t u^\pr)-f(x)-tx^*_1(u^\pr)].
\]
\end{definition}

\begin{definition}[\cite{RocWet}]\label{na.df-2hadsubd}
Let $f:\E\to\R\cup\{+\infty\}$ be an arbitrary proper real function. Suppose that $x\in{\dom}\, f$, $x^*_1\in\lsubd 1 f(x)$. The second-order lower subdifferential of Hadamard type of the function $f:\E\to\R\cup\{+\infty\}$ at the point
$x\in{\dom}\, f$ is defined by the following equality:
\[
\lsubd 2 f(x;x^*_1)=\{ x^*\in L^2(\E)\mid x^*(u)(u)\le\lhad 2  f ( x;x^*_1;u),\;\forall \; u\in\E\}.
\]
\end{definition}

\begin{definition}\label{critical}
Let the feasible point $\bar x$ be  a weak local minimizer, the functions $f$ and $g$ be Fr\'echet differentiable at $\bar x$.
A direction $u\in\R^s$ is callled critical at the point $\bar x$ for the problem (P) iff
\[
\nabla f(\bar x)u\in -C,\quad\nabla g(\bar x)u\in -K.
\]
\end{definition}


\begin{theorem}[\cite{na2015}]\label{na.th1}
Let the point $\bar x\in{\dom}\, f $ be a weak local minimizer of the vector problem {\rm (P)}. Then
\begin{equation}\label{na.15}
\lhad 1 F(\bar x;u)\ge 0,\quad\lhad 2  F(\bar x;0;u)\ge 0,\quad\textrm{ for all }u\in\E.
\end{equation}
\end{theorem}

\begin{theorem}\label{th3}
Let the point $\bar x$ be a weak local minimizer of the problem {\rm (P)}, where the functions $f$ and $g$ be Fr\'echet differentiable in a neighborhood of $\bar x$ and twice Fr\'echet differentiable at $\bar x$. Then
there exist $\lambda\in C^*$ and $\mu\in K^*$ such that $(\lambda,\mu)\ne (0,0)$ and
\be\label{16}
\lambda\cdot\nabla f(\bar x)+\mu\cdot\nabla g(\bar x)=0,\; \mu\cdot g(\bar x)=0.
\ee
\be\label{14}
[\lambda\cdot\nabla^2 f(\bar x)+\mu\cdot\nabla^2 g(\bar x)](u,u)\ge 0\quad\textrm{for every critical direction}\quad u\in\E
\ee
\end{theorem}
\begin{proof}
 It follows from Theorem \ref{na.th1} that
\[
\lhad 1 F(\bar x;u)\ge 0,\quad\lhad 2  F(\bar x;0;u)\ge 0\quad\textrm{ for all }u\in\E.
\]
Let $u\in\E$ be an arbitrary direction. Therefore
\[
\liminf_{t\downarrow 0,u^\pr\to u}\,\max_{(\lambda,\mu)\in\Lambda} [\lambda\cdot\frac{f(\bar x+t u^\pr)-f(\bar x)}{t}+\mu\cdot\frac{g(\bar x+t u^\pr)-g(\bar x)}{t}]\ge 0.
\] 
We can write from the last inequality that
\[
\liminf_{t\downarrow 0}\,\frac{1}{t}\max_{(\lambda,\mu)\in\Lambda}\{ [\lambda\cdot [f(\bar x+t u)-f(\bar x)]+\mu\cdot [{g(\bar x+t u)-g(\bar x)}]\}\ge 0.
\] 
Suppose that the lower limit is approached over the sequence $t_k$. Therefore,
\[
\lim_{k\to +\infty}\,\frac{1}{t_k}\max_{(\lambda,\mu)\in\Lambda} \lambda\cdot [f(\bar x+t_k u)-f(\bar x)]+\mu\cdot [g(\bar x+t_k u)-g(\bar x)]\ge 0.
\]
Suppose that the above maximum is attained over $\lambda_k$ and $\mu_k$. Then
\[
\lim_{k\to +\infty}\, [\lambda_k\frac{f(\bar x+t_k u)-f(\bar x)}{t_k}+\mu_k\frac{g(\bar x+t_k u)-g(\bar x)}{t_k}]\ge 0.
\]
Using that the set $\Lambda$ is compact, then we can suppose that $\lambda_k\to\lambda$, $\mu_k\to\mu$. The functions  $f$ and $g$ are  Fr\'echet differentiable. Then, taking the limits when $k$ approaches $+\infty$ we obtain that 
\be\label{13}
[\lambda\cdot\nabla f(\bar x)+\mu\cdot\nabla g(\bar x)]u\ge 0.
\ee
It follows from here that for every $u\in\R^s$ there exists $(\lambda,\mu)$ such that the inequality (\ref{13}) is satisfied.
Let us suppose additionally that the direction  $u\in\R^s$ is critical. Therefore
\[
\nabla f(\bar x)u\in -C,\quad\nabla g(\bar x)u\in -K.
\]
\[
\lambda_1\nabla f(\bar x)u\le 0,\quad\mu_1\nabla g(\bar x)u\le 0\quad\textrm{ for all }\quad\lambda_1\in C^*, \mu_1\in K^*.
\]
By (\ref{13}) we conclude that 
\be\label{15}
\lambda\nabla f(\bar x)u=0,\quad \mu\nabla g(\bar x)u=0.
\ee
By Theorem  \ref{na.th1} we have 
\[
0\le\lhad 2  F(\bar x;0;u)\le\liminf_{t\downarrow 0}\,\frac{1}{t^2}\max_{(\lambda,\mu)\in\Lambda}\{ [\lambda\cdot [f(\bar x+t u)-f(\bar x)]+\mu\cdot [{g(\bar x+t u)-g(\bar x)}]\}\le
\]
\[
\le\lim_{k\to +\infty}\,\frac{1}{t_k^2}\max_{(\lambda,\mu)\in\Lambda} \lambda\cdot [f(\bar x+t_k u)-f(\bar x)]+\mu\cdot [g(\bar x+t_k u)-g(\bar x)].
\]
The above maximum is attained over the same sequences $\lambda_k$ and $\mu_k$. Then
\be\label{k}
\lim_{k\to +\infty}\,\frac{1}{t_k^2} \{\lambda_k [f(\bar x+t_k u)-f(\bar x)]+\mu_k [g(\bar x+t_k u)-g(\bar x)]\}\ge 0.
\ee
Let us suppose additionally that $f$ and $g$ are twice differentiable at $\bar x$.
By Taylor expansion theorem 
\[
f(\bar x+t_k u)-f(\bar x)=t_k\nabla f(\bar x)u+(1/2)t_k^2\nabla^2 f(\bar x)(u,u)+o(t_k^2),
\]
\[
g(\bar x+t_k u)-g(\bar x)=t_k\nabla g(\bar x)u+(1/2)t_k^2\nabla^2 g(\bar x)(u,u)+o(t_k^2).
\]
It follows from $\lambda_k\in C^*$, $\mu_k\in K^*$, $\nabla f(\bar x)u\in -C$, $\nabla g(\bar x)u\in -K$ that
$\lambda_k\cdot\nabla f(\bar x)u\le 0$, $\mu_k\cdot\nabla g(\bar x)u\le 0$. Therefore,
\[
\lambda_k\cdot [f(\bar x+t_k u)-f(\bar x)]\le (1/2)t_k^2\lambda_k\cdot\nabla^2 f(\bar x)(u,u)+o(t_k^2),
\]
\[
\mu_k\cdot [g(\bar x+t_k u)-g(\bar x)]\le (1/2)t_k^2\mu_k\cdot\nabla^2 g(\bar x)(u,u)+o(t_k^2).
\]
Taking into account that $\lambda_k$ approaches $\lambda$, $\mu_k$ approaches $\mu$ when $k$ tends to $\infty$ we conclude from (\ref{k}) that 
inequality (\ref{14}) holds.

Denote by $D$ the set of all critical at $\bar x$ directions.
Let us consider the multivalued map $A(u)$, which associates any critical direction $u$ with the pair $(\lambda,\mu)$ satisfying equations (\ref{15}) and inequality (\ref{14}), that is  $(\lambda,\mu)\in A(u)$. Consider also the map 
\[
B(\lambda,\mu)=\{u\in D\,|\, u=-\lambda\cdot\nabla f(\bar x)-\mu\cdot\nabla g(\bar x)\}.
\]
The map $A(u)$ is nonempty, closed and convex-valued for every direction $u$. $B$ is linear. Then, by Kakutani fixed point theorem the composite map $B\cdot A$ has a fixed point. Therefore, there exists a direction $u_0$ such that $u_0\in B\cdot A(u_0)$. We prove as in the proof of Theorem \ref{th2} that
$u_0=0$. By the definition of the set $\Lambda$ we have $(\lambda,\mu)\ne (0,0)$. It follows from the definition of the function $F$ that $\mu\cdot g(\bar x)=0$. Therefore equations (\ref{16}) and inequality (\ref{14}) hold.

\end{proof}

\section{Sufficient conditions for a weakly globally efficient points}
In this section, we derive sufficient conditions for a weak global minimum in the problem (P).

Pseudoconvex scalar functions were introduced in optimization by Mangasarian \cite{man65}. We generalize them and apply in cone constrained vector optimization.
Sufficient conditions for global optimality in scalar problems with inequality constraints and pseudoconvex objective function were obtained initially in \cite{man65}. Global optimality conditions with quasiconcave objective function were derived before in \cite{arrent61}. In several articles were proved later more results containing sufficient conditions for global optimality.

\begin{definition}
A Fr\'echet differentiable vector function $f : X\to\R^n$, $X\subset\R^s$ is called pseudoconvex with respect to the cone $C$ iff the following implication is satisfied
\[
f(x)\in f(\bar x) - {\rm int(C)}, x\in X, \bar x\in X\quad\textrm{imply}\quad\nabla f(\bar x)(x-\bar x)\in - {\rm int(C)}.
\]
\end{definition}

\begin{definition}
A Fr\'echet differentiable scalar function $h : X\to\R$, $X\subset\R^s$ is called strictly pseudoconvex iff the following implication is satisfied
\[
h(x)\le h(\bar x), x\in X, \bar x\in X\quad\textrm{imply}\quad\nabla h(\bar x)(x-\bar x)<0.
\]
\end{definition}



\begin{theorem}{(First-order conditions)}
Let $\f$ and $\g$ be 
 Fr\'echet  differentiable vector-valued functions,
  $f$ be pseudoconvex with respect to the cone
 $C$. Suppose that first-order Fritz John conditions are satisfied at $\bar x$, that is there exists a pair 
$(\lambda,\mu)\in C^*\times K^*$  such that  $(\lambda,\mu)\ne (0,0)$ and
\be\label{foc}
\lambda\nabla f(\bar x)+\mu\nabla g(\bar x)=0.
\ee
Suppose also that $\mu\cdot g(\bar x)=0$ and the scalar function $\mu\cdot g(x)$ is strictly pseudoconvex, if $\mu\ne 0$. Then, $\bar x$ is a  weakly efficient global solution.
\end{theorem}
\begin{proof}
Assume in the contrary that $\bar x$ is not a global weakly efficient solution. Therefore there exists  another feasible point $x\in S$ with the property
 $f(x)\in f(\bar x)-{\rm int}(C)$. It follows from pseudoconvexity of $f$ that $\nabla f(\bar x)(x-\bar x)\in - {\rm int(C)}$. Therefore,
 $\lambda\cdot\nabla f(\bar x)(x-\bar x)<0$ if $\lambda\ne 0$. By $g(x)\in -K$ we obtain that
\[
\mu\cdot g(x)\le 0=\mu\cdot g(\bar x).
\]
By strict pseudoconvexity we conclude that $\mu\cdot\nabla g(\bar x)(x-\bar x)<0$ if $\mu\ne 0$. Thus, we obtained a contradiction to condition (\ref{foc}).
\end{proof}

Second-order pseudoconvex scalar functions were introduced in \cite{jmaa2008}. We generalize them and apply in cone constrained vector optimization.

\begin{definition}
A twice Fr\'echet differentiable vector function $f : X\to\R^n$, $X\subset\R^s$ is called second-order pseudoconvex with respect to the cone $C$ iff the following implications are satisfied
\[
f(x)\in f(\bar x) - {\rm int(C)}, x\in X, \bar x\in X\quad\textrm{imply}\quad\nabla f(\bar x)(x-\bar x)\in - C;
\]
\[
\begin{array}{c}
f(x)\in f(\bar x) - {\rm int(C)}, x\in X, \bar x\in X,\\
 \nabla f(\bar x)(x-\bar x)\in - (C\setminus\textrm{int}(C)) \quad\emph{}\textrm{imply}\quad (x-\bar x)\nabla^2 f(\bar x)(x-\bar x)\in - \textrm{int}(C).
\end{array}
\]
\end{definition}

\begin{definition}
A twice Fr\'echet differentiable scalar function $h : X\to\R$, $X\subset\R^s$ is called second order strictly pseudoconvex iff the following implications are satisfied
\[
h(x)\le h(\bar x), x\in X, \bar x\in X\quad\textrm{imply}\quad\nabla h(\bar x)(x-\bar x)\le 0;
\]
\[
h(x)\le h(\bar x), x\in X, \bar x\in X,\; \nabla h(\bar x)(x-\bar x) = 0\quad\textrm{imply}\quad (x-\bar x)\nabla^2 h(\bar x)(x-\bar x) < 0;
\]
\end{definition}


\begin{theorem}{(Second-order conditions)}
Let $\f$ and $\g$ be and twice Fr\'echet  differentiable  vector-valued functions at the feasible point $\bar x\in S$,  $f$ be second order pseudoconvex with respect to the cone $C$. Suppose that second order Fritz John conditions are satisfied at $\bar x$, that is there exists a pair 
$(\lambda,\mu)\in C^*\times K^*$  such that  $(\lambda,\mu)\ne (0,0)$ and
\be\label{soc1}
\lambda\cdot\nabla f(\bar x)+\mu\cdot\nabla g(\bar x)=0,
\ee
and
\begin{eqnarray}\label{soc2}
 \lambda\cdot (x-\bar x)\nabla^2 f(\bar x)(x-\bar x)+\mu\cdot (x-\bar x)\nabla^2 g(\bar x)(x-\bar x)\ge 0
\end{eqnarray}
 for every direction  $x-\bar x$ 
such that  

\noindent
\begin{equation*}
\begin{array}{c}
x\in S,\; f(x)\in f(\bar x)-{\rm int}(C), \\
\nabla f(\bar x)(x-\bar x)\in - (C\setminus {\rm int}(C)),\;\nabla g(\bar x)(x-\bar x)\in - (K\setminus {\rm int}(K)).
\end{array}
\end{equation*}

\noindent
Suppose additionally that $\mu\cdot g(\bar x)=0$ and the scalar function $\mu\cdot g(x)$ is second order strictly pseudoconvex if $\mu\ne 0$.
Then, $\bar x$ is a  weakly efficient global solution.
\end{theorem}
\begin{proof}
Assume in the contrary that $\bar x$ is not a weakly efficient global solution. Therefore there exists  another feasible point $x\in S$ with the property
 $f(x)\in f(\bar x)-{\rm int}(C)$. It follows from the second order pseudoconvexity of $f$ that $\nabla f(\bar x)(x-\bar x)\in - C$. Therefore, by $\lambda\in C^*$ we obtain that
 $\lambda\cdot\nabla f(\bar x)(x-\bar x)\le 0$. By $g(x)\in -K$ we obtain that
\[
\mu\cdot g(x)\le 0=\mu\cdot g(\bar x).
\]
By second order strict pseudoconvexity we conclude that $\mu\cdot\nabla g(\bar x)(x-\bar x)\le 0$ if $\mu\ne 0$.
It follows from equation (\ref{soc1}) by $\lambda\cdot\nabla f(\bar x)(x-\bar x)\le 0$ and $\mu\cdot\nabla g(\bar x)(x-\bar x)\le 0$ that
$\lambda\cdot\nabla f(\bar x)(x-\bar x) = 0$ and $\mu\cdot\nabla g(\bar x)(x-\bar x) = 0$.
 

If $\lambda\ne 0$, then $\nabla f(\bar x)(x-\bar x)\in -(C\setminus\textrm{int}(C))$. By the second order pseudoconvexity of $f$ we obtain that
$(x-\bar x)\nabla^2 f(\bar x)(x-\bar x)\in -\textrm{int}(C)$. Therefore, $\lambda\cdot (x-\bar x)\nabla^2 f(\bar x)(x-\bar x)<0$.

If $\mu\ne 0$, then by $\mu\cdot g(x)\le\mu\cdot g(\bar x)$, second order strict pseudoconvexity and $\mu\cdot\nabla g(\bar x)(x-\bar x) = 0$ we conclude that $\mu\cdot (x-\bar x)\nabla^2 g(\bar x)(x-\bar x)<0$. Thus $(\lambda,\mu)\ne (0,0)$ contradicts the inequality 
(\ref{soc2}).
\end{proof}

\section{Sufficient conditions for a weakly locally efficient points}
In this section, we prove sufficient conditions for a weak isolated local minimum in the problem (P).

The notion isolated local minimum of order 1 and 2 was introduced in \cite{aus84}. Somewhere isolated minimizers of order 1 or 2 are named strict minimizers of order 1 or 2 \cite{JN08}. We generalize these notions to cone constrained vector problems and apply these notions in cone constrained optimization.

\begin{definition}
A point $\bar x\in S$ is called a (first-order) weak isolated local minimizer iff there exist $\lambda\in C^*$, $\lambda\ne 0$ and a number $\epsilon>0$ such that
\be\label{22}
\lambda\cdot f(x)\ge\lambda\cdot f(\bar x)+\epsilon\norm{x-\bar x},\quad\forall x\in N_\epsilon(\bar x)\cap S
\ee
where $N_\epsilon(\bar x)$ is a round neighborhood of $\bar x$ with a radius $\epsilon$.
\end{definition}

\begin{theorem}
Let $\bar x$ be a feasible point. Suppose that there exists $\bar\lambda\in C^*$, $\bar\lambda\ne 0$ and $\bar\mu\in K^*$  with   the  property that the functions $f$, $g$ are locally Lipschitz, $\bar\mu g(\bar x)=0$ and 
\be\label{23}
\bar\lambda\cdot\ldini f(\bar x,u)+\bar\mu\cdot\ldini g(\bar x,u) > 0\quad\textrm{for every direction}\quad u\in\E,\; u\ne 0.
\ee
Then, $\bar x$ is a (first-order) weak isolated local minimizer.
\end{theorem}
\begin{proof}
Suppose the contrary that $\bar x$ is not a weak isolated local minimizer. Therefore, condition (\ref{22}) fails for every $\lambda\in C^*$, $\lambda\ne 0$ and every $\epsilon>0$. It follows from here that for every sequence of positive numbers $\epsilon_n$ converging to zero there is a sequence 
$x_n\in S\cap N_{\epsilon_n}(\bar x) $ such that
\[
\bar\lambda\cdot f(x_n)-\bar\lambda\cdot f(\bar x)<\epsilon_n\norm{x_n-\bar x}.
\]
Let us denote $t_n=\norm{x_n-\bar x}$, $u_n=(x_n-\bar x)/t_n$. Without loss of generality we can suppose that the sequence $u_n$ is convergent, that is $u_n\to u$, $u\ne 0$ and $t_n\to 0$. Therefore, 
\[
\bar\lambda\cdot f(\bar x+t_n u_n)-\bar\lambda\cdot f(\bar x)<\epsilon_n t_n.
\]
Taking into account that the function $\bar\lambda\cdot f$ is locally Lipschitz, we obtain that
\[
\lim_{t\to 0,u^\pr\to u}\frac{\bar\lambda\cdot f(\bar x+t u^\pr)-\bar\lambda\cdot f(\bar x+t u)}{t} = 0.
\]
Therefore, 
\[
\liminf_{n\to +\infty}\frac{\bar\lambda\cdot f(\bar x+t_n u)-\bar\lambda\cdot f(\bar x)}{t_n}\le 0.
\]
It follows from here that 
\be\label{8}
\bar\lambda\cdot\ldini f(\bar x,u)\le 0.
\ee

It follws from $g(x_n)\in -K$ and   $\bar\mu\in K^*$ that $\bar\mu\, g(x_n)\le 0$. By $\bar\mu\, g(\bar  x)=0$ we obtain
\[
\liminf_{t\to 0}\frac{\bar\mu\cdot g(\bar x+t u)-\bar\mu\cdot g(\bar x)}{t}=\liminf_{t\to 0,u^\pr\to u}\frac{\bar\mu\cdot g(\bar x+t u^\pr)-\bar\mu\cdot g(\bar x)}{t}\le 
\]
\[
\liminf_{n\to\infty}\frac{\bar\mu\cdot g(\bar x+t_n u_n)-\bar\mu\cdot g(\bar x)}{t_n}\le 0.
\]
Therefore
\be\label{9}
\bar\mu\cdot\ldini g(\bar x,u)\le 0
\ee
Inequalities (\ref{8}) and (\ref{9}) contradict to condition (\ref{23}).
\end{proof}

\begin{definition}
A point $\bar x\in S$ is called a second-order weak isolated local minimizer iff there exist $\lambda\in C^*$, $\lambda\ne 0$ and a number $\epsilon>0$ such that
\be\label{21}
\lambda\cdot f(x)\ge\lambda\cdot f(\bar x)+\epsilon\norm{x-\bar x}^2,\quad\forall x\in N_\epsilon(\bar x)\cap S
\ee
where $N_\epsilon(\bar x)$ is a round neighborhood of $\bar x$ with a radius $\epsilon$.
\end{definition}
\begin{theorem}
Let $\bar x$ be a feasible point and the function $f$ be twice continuously differentiable. Suppose that there exists $\bar\lambda\in C^*$, $\bar\lambda\ne 0$ and $\bar\mu\in K^*$  with   the following property:
\be\label{20}
\begin{array}{c}
\bar\lambda\cdot\nabla f(\bar x) + \bar\mu\cdot\nabla g(\bar x) = 0,\\
\bar\lambda\cdot\nabla^2 f(\bar x)(u,u) + \bar\mu\cdot\nabla^2 g(\bar x)(u,u) > 0\quad\textrm{for every critical direction}\quad u\in\E,\; u\ne 0.
\end{array}
\ee
$\bar\mu\cdot g(\bar x) = 0$. Then, $\bar x$ is a second-order weak isolated local minimizer.
\end{theorem}
\begin{proof}
Suppose the contrary that $\bar x$ is not a second-order weak isolated local minimizer. Therefore, condition (\ref{21}) fail for every $\lambda\in C^*$, $\lambda\ne 0$ and every $\epsilon>0$. It follows from here that for every sequence of positive numbers $\epsilon_n$ converging to zero there is a sequence $x_n\in S\cap N_{\epsilon_n}(\bar x) $ such that
\[
\bar\lambda\cdot f(x_n)-\bar\lambda\cdot f(\bar x)<\epsilon_n\norm{x_n-\bar x}^2.
\]
Let us denote $t_n=\norm{x_n-\bar x}$, $u_n=(x_n-\bar x)/t_n$. Without loss of generality we can suppose that $u_n$ is convergent, that is $u_n\to u$, $u\ne 0$ and $t_n\to 0$. Therefore,
\[
\bar\lambda\cdot f(\bar x+t_n u_n)-\bar\lambda\cdot f(\bar x)<\epsilon_n t_n^2.
\]
By Taylor expansion theorem 
\[
f(\bar x+t_n u_n)-f(\bar x)=t_n\nabla f(\bar x)u_n+(1/2)t_n^2\nabla^2 f(\bar x)(u_n,u_n)+o(t_n^2).
\]
\be\label{30}
\bar\lambda\cdot f(\bar x+t_n u_n)-\bar\lambda\cdot f(\bar x)=t_n\bar\lambda\cdot\nabla f(\bar x)u_n+(1/2)t_n^2\bar\lambda\cdot\nabla^2 f(\bar x)(u_n,u_n)+o(t_n^2).
\ee
\[
g(\bar x+t_n u_n)-g(\bar x)=t_n\nabla g(\bar x)u_n+(1/2)t_n^2\nabla^2 g(\bar x)(u_n,u_n)+o(t_n^2).
\]
\be\label{31}
\bar\mu\cdot g(\bar x+t_n u_n)-\bar\mu g(\bar x) = t_n\bar\mu\cdot \nabla g(\bar x)u_n+(1/2)t_n^2\bar\mu\cdot \nabla^2 g(\bar x)(u_n,u_n)+o(t_n^2).
\ee
On the other hand, we can take into account that $\bar\lambda\cdot f(\bar x+t_n u_n)-\bar\lambda\cdot f(\bar x)<\epsilon_n t_n^2$ and $\bar\mu\cdot g(x_n)\le 0$ by $g(x_n)\in -K$,   $\bar\mu\in K^*$ and $\bar\mu g(\bar x)=0$. Therefore, by adding  (\ref{30}) and (\ref{31}) we obtain
\[
t_n\bar\lambda\cdot\nabla f(\bar x)u_n+t_n\bar\mu\cdot\nabla g(\bar x)u_n+(1/2)t_n^2\bar\lambda\cdot\nabla^2 f(\bar x)(u_n,u_n)+(1/2)t_n^2\bar\mu\cdot\nabla^2 g(\bar x)(u_n,u_n)+o(t_n^2)\le \epsilon_n t_n^2.
\]
We conclude from the equation in (\ref{20}) that
\[
t_n^2\bar\lambda\cdot\nabla^2 f(\bar x)(u_n,u_n)+t_n^2\bar\mu\cdot\nabla^2 g(\bar x)(u_n,u_n)+o(t_n^2)\le 2\epsilon_n t_n^2.
\]
Simplifying this expression and taking the limits when $n\to\infty$ we obtain
\[
\bar\lambda\cdot\nabla^2 f(\bar x)(u,u)+\bar\mu\cdot\nabla^2 g(\bar x)(u,u)\le 0,
\]
which contradicts the inequality in (\ref{20}),  because the direction $u$ is critical.
\end{proof}

\end{document}